\documentclass[12pt, reqno]{amsart}
\usepackage{mathrsfs}
\usepackage{amsfonts}
\usepackage{cite}
\usepackage{amsmath}
\usepackage{amsthm}
\usepackage{amssymb}
\usepackage{cases}
\numberwithin{equation}{section}
\newtheorem{thm}{Theorem}[section]

\newtheorem{lemma}{Lemma}[section]

\newtheorem{cor}{Corollary}[section]

\newtheorem{defi}{Definition}[section]

\usepackage[%
pdfstartview=FitH, %
colorlinks=true, %
linkcolor=blue, %
citecolor=red]{hyperref}

\begin{document}
\title[Notes on the functional LYZ ellipsoid] {Notes on the functional LYZ ellipsoid}

\author{Niufa Fang}

\address{\parbox[l]{1\textwidth}{Chern Institute of Mathematics, Nankai University, Tianjin 300071, China\\
School of Mathematics and
Statistics, Southwest University, Chongqing 400715, China }}
\email{fangniufa@nankai.edu.cn}

\author{Jiazu Zhou*}
\address{\parbox[l]{1\textwidth}{School of Mathematics and
Statistics, Southwest University, Chongqing 400715, China\\
College of Science, Wuhan University of Science and Technology, Wuhan, Hubei 430081, China}}
\email{zhoujz@swu.edu.cn}

%\date{2019/10/14}
\subjclass[2000]{52A20, 52A30} \keywords{Log-concave functions;
affine Sobolev inequality, $L_p$ John ellipsoids, functional LYZ
ellipsoid of log-concave functions.}

\thanks{The first author is supported in part by China Postdoctoral Science  Foundation (No.2019M651001). }
\thanks{*The corresponding author is supported in part by NSFC (No. 11671325)
}

\maketitle

\begin{abstract}
We define the functional LYZ ellipsoid of log-concave functions.
Then we give notes appended to \cite{fangzhou}.
\end{abstract}

\vskip 1cm

\section{Introduction and Preliminaries}

Ellipsoids are important objections in convex geometry  that
characterize  equality conditions of many known affine isoperimetric
inequalities, such as the Petty projection inequality, the
Busemann-Petty centroid inequality and its polar, the
Blaschke-Santal\'o inequality.

The well-known \emph{John ellipsoid}, an ellipsoid of  maximal
volume contained in a related   convex body, is a very important
geometric object in convex geometry (cf. \cite{ball},
\cite{pisier}). Two important facts  concerning the John ellipsoid
are John's inclusion inequality and Ball's volume-ratio inequality
\cite{ball}. In 2000, Lutwak, Yang and Zhang introduced a new
ellipsoid,  the \emph{LYZ ellipsoid} \cite{LYZ2000}. Lutwak, Yang
and Zhang \cite{LYZ2005} developed the $L_p$ John ellipsoid.
Although $L_p$ John ellipsoids may not be contained in the original
convex body, $L_p$ John ellipsoids  satisfy Ball's volume ratio
inequality.  The $L_p$ John ellipsoid provides a unified treatment
for several fundamental objects in convex geometry.  If the John
point, the center of John ellipsoid,   is at the origin, then
$L_{\infty}K$ John ellipsoid  is precisely the classical John
ellipsoid and the $L_1$ John ellipsoid and the $L_2$ John ellipsoid
are, respectively, the Petty and  LYZ ellipsoids.
  Orlicz John ellipsoids \cite{zou}
  are generalizations of  the classical John ellipsoid and the
$L_p$ John ellipsoids. Recently the John ellipsoid of log-concave
functions was defined by Alonso, Merino, Jim\'enez and Villa
\cite{Alonso}, the LYZ ellipsoid for log-concave functions was
defined in \cite{fangzhou} and  Loewner (ellipsoid) function for a
log-concave function was introduced by Li, Schuett and
Werner\cite{LSW}

Let $\mathcal{K}_o^n$ denotes the set of convex bodies in
$\mathbb{R}^n$ containing the origin in their interior.
 The radial  function, $\rho_K$,  of a convex body $K$ is
defined as $\rho_K(x)=\max\{t\geq 0:  t x\in K\}$, $x\in
\mathbb{R}^n\setminus\{0\}$.  The   LYZ ellipsoid was  defined as
\cite{LYZ2000}
\begin{eqnarray}\label{LYZ's definition}
\rho_{\Gamma_{-2}^{}K}^{2}(u)=\frac{1}{V(K)}\int_{S^{n-1}}|\langle u, v\rangle|^2dS_2(K,v),
\end{eqnarray}
where $K\in \mathcal{K}_o^n$  with volume $V(K)$, and $S_{2}(K,\cdot)$ denotes the $L_2$ surface area measure of $K$.

To generalize  formula (\ref{LYZ's definition}) to the case of
log-concave functions, the corresponding $L_2$ surface area measure
of a log-concave function should be defined. The surface area
measure $\mu_f$ ($=\nabla (-\log f)_{\sharp} (f \mathcal{H}^{n})$,
where $\mathcal{H}^{n}$ is the $n$-dimensional Hausdorff measure) of
log-concave function $f$, introduced by Colesanti and
Fragal$\grave{\text{a}}$ \cite{Colesanti}, implies  $L_2$ surface
area measure of convex bodies.  It is not  the $L_2$ surface area
measure of log-concave function $f$. The main reason is  the
corresponding $L_2$-Minkowski addition of log-concave functions.

Recall the $L_2$ surface area measure and the classical surface area
measure
\begin{eqnarray*}\label{}
\frac{dS_{2}(K,\cdot)}{dS_{}(K,\cdot)}=h_K^{-1}(\cdot).
\end{eqnarray*}
The LYZ ellipsoid of log-concave functions can be defined as:

\begin{defi}\label{new definition}
Let $f$ be an integrable log-concave function in $\mathbb{R}^n$. Let
$h_f(x)=\sup_{y\in \mathbb{R}^n}\{\langle x,y\rangle+\log f(y)\}$
denote the support function of $f$ and
$J(f)=\int_{\mathbb{R}^n}fdx$. The log-concave function
$\Gamma_{-2}f: \mathbb{R}^n\rightarrow [0,\infty)$
\begin{eqnarray}\label{definition}
-\log \Gamma_{-2}f(x)=\frac{n}{4 J(f)}\int_{\mathbb{R}^n}|\langle x, y\rangle|^2h_f(y)^{-1}d\mu_f(y),
\end{eqnarray}
is called the functional LYZ ellipsoid of $f$.
\end{defi}

The right hand side of (\ref{definition}) is a continuous convex
function, and  the constant $\frac{n}{4J(f)}$ in (\ref{definition})
is chosen such that $\Gamma_{-2}^{}\gamma(x)=\gamma(x) \ \
(=e^{-\frac{\|x\|^2}{2}})$, the Gaussian function.  The
$\Gamma_{-2}f$ includes the LYZ ellipsoid (Lemma \ref{include
geometric case}).

Let $\varphi: \mathbb{R}^n\rightarrow \mathbb{R}\cup \{+\infty\}$.   If for every $x,y\in\mathbb{R}^n$ and $\lambda\in[0,1]$,
$$
\varphi((1-\lambda)x+\lambda y)\leq (1-\lambda)\varphi(x)+\lambda \varphi(y),
$$
then $\varphi$ is a convex function. Let
$$
\text{dom}(\varphi)=\{x\in \mathbb{R}^n: \varphi(x)\in \mathbb{R}\}.
$$
By the convexity of $\varphi$, $\text{dom}(\varphi)$ is a convex set. We say that $\varphi$ is \emph{proper} if $\text{dom}(\varphi)\neq \emptyset$.
The \emph{Legendre conjugate} of $\varphi$ is the convex function defined by
\begin{eqnarray}\label{Fenchel conjugate}
\varphi^*(y)=\sup_{x\in\mathbb{R}^n}\left\{\langle x,y\rangle-\varphi(x)\right\}\quad\quad\forall y\in\mathbb{R}^n.
\end{eqnarray}
Clearly, $\varphi(x)+\varphi^*(y)\geq \langle x,y\rangle$ for all $x, y\in\mathbb{R}^n $, there is an equality if and only if  $x\in\text{dom}(\varphi)$ and $y$ is in the subdifferential of $\varphi$ at $x$. Hence,
 $$
 \varphi^*(\nabla \varphi(x))+\varphi(x)= \langle x,\nabla \varphi(x)\rangle.
 $$

The convex function $\varphi:\mathbb{R}^n\rightarrow \mathbb{R}\cup \{+\infty\}$ is  \emph{lower semi-continuous},
if the subset $\{x\in \mathbb{R}^n: \varphi(x)>t\}$ is an open set for any $t\in(-\infty,+\infty]$.
If $\varphi$ is a lower semi-continuous  convex function, then also $\varphi^*$ is a lower semi-continuous  convex function, and $\varphi^{**}=\varphi$.

On the class of convex functions from $\mathbb{R}^n$ to $\mathbb{R}\cup \{+\infty\}$, we consider the operation of \emph{infimal convolution},
defined by

\begin{eqnarray}\label{infimal convolution}
\varphi\Box \psi(x)=\inf_{y\in\mathbb{R}^n}\{\varphi(x-y)+\psi(y)\} \quad \forall x\in\mathbb{R}^n,
\end{eqnarray}
and the following \emph{right scalar multiplication} by a nonnegative real number $\alpha$:
\begin{eqnarray}\label{right scalar multiplication}
(\varphi\alpha)(x)=\alpha \varphi\left(\frac{x}{\alpha}\right), \quad\text{for}\quad \alpha>0.
\end{eqnarray}

In this paper, we consider   log-concave functions in $\mathbb{R}^n$:
\begin{eqnarray*}
f:\mathbb{R}^n\rightarrow \mathbb{R},\quad f=e^{-\varphi},
\end{eqnarray*}
where $\varphi:\mathbb{R}^n\rightarrow \mathbb{R}\cup \{+\infty\}$ is convex.

A log-concave function is degenerate if it  vanishes almost everywhere in $\mathbb{R}^n$.  A non-degenerate, log-concave  function $e^{-\varphi}$ is integrable on $\mathbb{R}^n$ if and only if \cite{Cordero-ErausquinKlartag}
$$
\lim_{\|x\|\rightarrow+\infty}\varphi(x)=+\infty.
$$
Any log-concave function is differentiable almost everywhere in $\mathbb{R}^n$.

Let
\begin{eqnarray}\label{}
\mathcal{L}&=&\left\{\varphi:\mathbb{R}^n\rightarrow \mathbb{R}\cup \{+\infty\} \Big| \quad
\varphi \text{ proper, convex, }\lim_{\|x\|\rightarrow+\infty}\varphi(x)=+\infty\right\},\\
\mathcal{A}&=&\left\{f:\mathbb{R}^n\rightarrow \mathbb{R} \Big| \quad  f=e^{-\varphi},\varphi\in \mathcal{L}\right\}.
\end{eqnarray}
The \emph{total mass functional} of $f$ is defined as
\begin{eqnarray}
J(f)=\int_{\mathbb{R}^n}f(x)dx.
\end{eqnarray}

There are important connections between convex bodies and log-concave functions.
The Gaussian function
$$\gamma(x)=e^{-\frac{\|x\|^2}{2}}$$
plays the role in analysis of log-concave functions as the ball does in geometry of convex bodies, and
$$J(\gamma)=(2\pi)^{\frac{n}{2}}=c_n.$$

  If $f\in \mathcal{A}$, the polar function, $f^{\circ}$, of $f$ is defined as
\begin{eqnarray}\label{dual function}
f^{\circ}=e^{-\varphi^*}.
\end{eqnarray}
The \emph{support function} of log-concave function $f=e^{-\varphi}$
is  defined in \cite{Rotem1}.
%by Rotem \cite{Rotem1},
\begin{eqnarray*}\label{}
h_{f}(x)=\varphi^*(x).
\end{eqnarray*}
%This is a proper generalization, in the sense that $h_{\chi_K}=h_K$.

Let $f=e^{-\varphi},g=e^{-\psi}$ and  $\alpha,\beta>0$. The ``Minkowski addition", $\alpha\cdot f\oplus\beta\cdot g$, of $f$ and $g$ is defined by
\begin{eqnarray}\label{addtion on log-concave function}
\alpha\cdot f\oplus\beta\cdot g=e^{-[(\varphi\alpha)\square(\psi\beta)]}.
\end{eqnarray}
The support function of $\alpha\cdot f\oplus\beta\cdot g$ satisfies
\begin{eqnarray}\label{definition of the support function of sum of log-concave functions}
h_{\alpha\cdot f\oplus\beta\cdot g}(x)=\alpha h_f(x)+\beta h_g(x).
\end{eqnarray}
In particular,
\begin{eqnarray}
h_{\alpha\cdot f}(x)=\alpha h_f(x).
\end{eqnarray}

Let $f,g\in \mathcal{A}$. Whenever the following limit exists
\begin{eqnarray}\label{definition of the first variation}
\lim_{t\rightarrow 0^{+}}\frac{J(f\oplus t\cdot g)-J(f)}{t},
\end{eqnarray}
denote it by $\delta J(f,g)$, \emph{the first variation} of $J$ at $f$ along $g$.
For $f,g\in \mathcal{A}$,
 set
\begin{eqnarray*}
\overline{ \delta J}(f,g)=\frac{ \delta J(f,g)}{J(f)}.
\end{eqnarray*}

\section{Petty projection inequality for log-concave functions}

\begin{lemma} \label{equivalent of two problem} Suppose $f\in\mathcal{A}$.  \\
(1) If $\gamma_T$ is a Gaussian function that is an $\overline{S_{\log}}$ solution for $f$, then
\begin{eqnarray*}
\frac{J(f)}{\delta J(f,\gamma_T)}\cdot\gamma_T
\end{eqnarray*}
is an $S_{\log}$ solution for $f$.
\\ (2) If $\gamma_{T}$ is a  Gaussian function that is an $S_{\log}$ solution for $f$, then
\begin{eqnarray*}
\left(\frac{c_n}{J(\gamma_{T})}\right)^{\frac{2}{n}}\cdot\gamma_{T}
\end{eqnarray*}
is an $\overline{S_{\log}}$ solution for $f$.
\end{lemma}
\begin{proof}
(1) Since $\gamma_T$ is a  Gaussian function that is an $\overline{S_{\log}}$ solution for $f$, Lemma 4.1 (2) in \cite{fangzhou} tells  that $J(\gamma_T)=c_n$, and for any $P\in {\rm GL}(n)$ with $J(\gamma_ {P})= c_n$,
\begin{eqnarray}\label{equivalent of two problem1}
\overline{ \delta J}(f,\gamma_T)\leq \overline{ \delta J}(f,\gamma_ P).
\end{eqnarray}

By directly computing we have
\begin{eqnarray*}
\overline{ \delta J}\left(f,\frac{J(f)}{\delta J(f,\gamma_T)}\cdot\gamma_T\right)=1.
\end{eqnarray*}
From (\ref{equivalent of two problem1}), and $J(\gamma_ {P})=J(\gamma_ {T})= c_n$, we have
\begin{eqnarray*}
J\left(\frac{J(f)}{\delta J(f,\gamma_P)}\cdot\gamma_P\right)\leq J\left(\frac{J(f)}{\delta J(f,\gamma_T)}\cdot\gamma_T\right).
\end{eqnarray*}
Lemma 4.1
(1) in \cite{fangzhou} guarantees  that the Gaussian function $
\frac{J(f)}{\delta J(f,\gamma_T)}\cdot\gamma_T
$ is an $S_{\log}$ solution for $f$.

(2) Let $\gamma_{T}$  be an $S_{\log}$ solution for $f$.  Then Lemma 4.1 (1) in \cite{fangzhou} tells  that
\begin{eqnarray*}
\overline{ \delta J}(f,\gamma_{T})=1,
\end{eqnarray*}
and for any $P\in {\rm GL}(n)$ with $ \overline{\delta J}(f,\gamma_ {P})=1$,
\begin{eqnarray}\label{equivalent of two problem2}
\frac{J(\gamma_ {T})}{c_n}\geq \frac{J(\gamma_ {P})}{c_n}.
\end{eqnarray}
By (\ref{equivalent of two problem2}), $ \overline{\delta J}(f,\gamma_ {P})= \overline{\delta J}(f,\gamma_ {T})=1$, we obtain
\begin{eqnarray*}
\overline{ \delta J}\left(f,\left(\frac{c_n}{J(\gamma_{P})}\right)^{\frac{2}{n}}\cdot\gamma_{P}\right)
&=&\left(\frac{c_n}{J(\gamma_{P})}\right)^{\frac{2}{n}}\overline{ \delta J}(f,\gamma_{P})\\
&\geq& \left(\frac{c_n}{J(\gamma_{T})}\right)^{\frac{2}{n}}\overline{ \delta J}(f,\gamma_{T})\\
&=&\overline{ \delta J}\left(f,\left(\frac{c_n}{J(\gamma_{T})}\right)^{\frac{2}{n}}\cdot\gamma_{T}\right).
\end{eqnarray*}
  Since
\begin{eqnarray*}
J\left(\left(\frac{c_n}{J(\gamma_{T})}\right)^{\frac{2}{n}}\cdot\gamma_{T}\right)=c_n,
\end{eqnarray*}
Lemma 4.1 (2) in\cite{fangzhou} ensures  that the Gaussian function $\left(\frac{c_n}{J(\gamma_{T})}\right)^{\frac{2}{n}}\cdot\gamma_{T}$
 is an $\overline{S_{\log}}$ solution for $f$.
\end{proof}

\begin{lemma}\label{existence}
There exists a  solution to Problem $S_{\log}$.
\end{lemma}
\begin{proof}
 Given a  function $\gamma_T=e^{-\frac{\|Tx\|^2}{2}}$ (where $T\in \text{GL}(n)$ is a positive definite symmetric matric),
 we can write it as $\gamma_T=e^{-\frac{\|x\|_{E^{\circ}}^2}{2}}$,  where $E=T^{t}B$ is an origin-centered ellipsoid.
 There exists a $v_E\in S^{n-1}$ such that the diameter of $E$ satisfies $\text{diam }(E) \frac{|\langle v_E,x\rangle|}{2}\leq \|x\|_{E^{\circ}}$.
 Then, we have
 \begin{eqnarray*}
\frac{ \text{diam }(E^{\circ})^2}{8} \min_{v\in S^{n-1}}\int_{\mathbb{R}^n}|\langle v,x\rangle|^2d\mu_f(x)&\leq&\int_{\mathbb{R}^n}\frac{\|x\|_E^2}{2}d\mu_f(x)\\
 &=&\delta J(f,\gamma_ T)\\
 &\leq& J(f).
\end{eqnarray*}
Since $\|x\|_E<\text{diam }(E^{\circ})\|x\|$, therefore the upper
bound of  the convex function $\left(\frac{\|Tx\|^2}{2}\right)^*$ is
depended only on $f$. According to  Theorem  10.9 in \cite{R}, there
exist a convex function $\tilde{\varphi}$ such that the Legendre
conjugate of a maximizing sequence of  Gaussian functions for
Problem $S_{\log}$ converge to $\tilde{\varphi}$. By  Lemma 6 in
\cite{ColesantiLudwigMussnig1} or Theorem 11.34 in \cite{RW}, a
maximizing sequence of  Gaussian functions for Problem $S_{\log}$
converge to $\tilde{\varphi}^*$. By the dominated convergence
theorem, there exists a solution to Problem $S_{\log}$.
 \end{proof}

 Theorem 4.2 in \cite{fangzhou} can be rewritten  in a simple form.

\begin{thm}\label{main theorem2}
Let $f\in\mathcal{A}'$. Then problem $S_{\log}$ has a unique solution. Moreover, a Gaussian function $\gamma_T$ solves $S_{\log}$ if and only if it satisfies
\begin{eqnarray}\label{main theorem2-1}
h_{\gamma_T^{\circ}}(y)=\frac{n}{4J(f)}\int_{\mathbb{R}^n}|\langle x, y\rangle|^2d\mu_f(x),
\end{eqnarray}
for all $ y\in \mathbb{R}^n$.
\end{thm}

\begin{proof}
From Lemma \ref{equivalent of two problem}, we note that if $\gamma_T$  solves $S_{\log}$,
then $\left(\frac{c_n}{J(\gamma_{T})}\right)^{\frac{2}{n}}\cdot\gamma_{T}$ solves $\overline{S_{\log}}$.
Since
\begin{eqnarray}\label{step5-1}
\delta J\left(f,\left(\frac{c_n}{J(\gamma_{T})}\right)^{\frac{2}{n}}\cdot\gamma_T\right)
&=&\left(\frac{c_n}{J(\gamma_{T})}\right)^{\frac{2}{n}}\delta J(f,\gamma_{T}) \nonumber\\
&=&J(f)\left(\frac{c_n}{J(\gamma_{T})}\right)^{\frac{2}{n}},
\end{eqnarray}
and
\begin{eqnarray}\label{step5-2}
h\left(\left(\left(\frac{c_n}{J(\gamma_{T})}\right)^{\frac{2}{n}}\cdot\gamma_{T}\right)^{\circ},y\right)
&=&\left(\frac{c_n}{J(\gamma_{T})}\right)^{\frac{2}{n}}\cdot h_{\gamma_T^{\circ}}(y)\nonumber\\
&=&\left(\frac{c_n}{J(\gamma_{T})}\right)^{\frac{2}{n}}h_{\gamma_T^{\circ}}\left(\frac{y}{\left(\tfrac{c_n}{J(\gamma_{T})}\right)^{\frac{2}{n}}}\right),
\end{eqnarray}
the formula (\ref{main theorem2-1}) follows from  Theorem 4.1 in \cite{fangzhou}, (\ref{step5-1}) and  (\ref{step5-2}).
\end{proof}

 Let $f\in\mathcal{A}$. The new log-concave function $\Pi f$ of $f=e^{-\varphi}$ is defined by

 \begin{eqnarray}\label{definition of the projection function of log-concave functions0}
h(\Pi f, y)=h_{\Pi f}(y)=\frac{1}{2}\int_{\mathbb{R}^n}|\langle x, y\rangle| d\mu_f(x),
\end{eqnarray}
for   $y\in \mathbb{R}^n$.  In particular,  if  $f=e^{-\|x\|_K}$, then

\begin{eqnarray}
h_{\Pi f}(x)= \Gamma(n) h_{\Pi  K}(x),
\end{eqnarray}
for $x\in \mathbb{R}^n$ and $K\in \mathcal{K}_o^n$.

We have the following Petty projection inequality for log-concave
functions.

\begin{thm}\label{Petty projection inequality of log-concave functions}
 Let $f\in \mathcal{A}$ and $\Pi^{\circ} f$ denote the polar of $\Pi f$.  Then
\begin{eqnarray*}\label{Petty projection inequality of log-concave functions 1}
\int_{\mathbb{R}^n}\|\nabla f(x)\|dx\geq n\omega_n^{\frac{1}{n}}
\left[ \frac{\omega_{n-1}^n}{ \omega_{n}^n\Gamma(n+1)}J(\Pi^{\circ}f)\right]^{-\frac{1}{n}}\geq n^{}\omega_n^{1/n}\left( \int_{\mathbb{R}^n} |
f(x) |^{\frac n{n-1}} dx\right)^{\frac{n-1}n}.
\end{eqnarray*}
The first inequality holds with  equality sign if $f$ is  the characteristic function of  a ball and the second
inequality holds with  equality sign if $f$ is  the characteristic function of  an ellipsoid.
\end{thm}
\begin{proof}

By  the definition of $\Pi^{\circ}f$,  integral of  polar coordinates with respect to the sphere Lebesgue measure, Jensen's  inequality and Fubini's Theorem,
we have
\begin{eqnarray}\label{Petty projection inequality of log-concave functions 2}
\int_{\mathbb{R}^n}\Pi^{\circ}f(x)dx&=&\int_{\mathbb{R}^n}\text{exp}\left[-\frac{1}{2}\int_{\mathbb{R}^n}|\langle\nabla \varphi, y\rangle| fdx\right]dy  \nonumber  \\
&=& \int_{0}^{\infty}r^{n-1} \int_{S^{n-1}}\text{exp}\left[-\frac{r}{2}\int_{\mathbb{R}^n}|\langle\nabla f(x), u\rangle| dx\right]dudr  \nonumber\\
& \geq&  n\omega_n \int_{0}^{\infty} r^{n-1}\text{exp}\left[- \frac{1}{ n\omega_n}\int_{S^{n-1}}\frac{r}{2}
          \int_{\mathbb{R}^n}|\langle\nabla f, u\rangle| dxdu  \right]dr \nonumber\\
& =&  n\omega_n \int_{0}^{\infty} r^{n-1}\text{exp}\left[ -\frac{\omega_{n-1}}{ n\omega_n}r\int_{\mathbb{R}^n}\|\nabla f\| dx  \right]dr \nonumber\\
& =& \left[ \frac{\omega_{n-1}}{ n\omega_n}\int_{\mathbb{R}^n}\|\nabla f\| dx  \right]^{-n} n\omega_n \int_{0}^{\infty} r^{n-1}e^{-r}dr   \nonumber \\
&=& n\omega_n \Gamma(n)\left( \frac{\omega_{n-1}}{ n\omega_n}  \right)^{-n}    \left[\int_{\mathbb{R}^n}\|\nabla f\| dx  \right]^{-n}.
\end{eqnarray}
Equality holds if and only if $
\int_{\mathbb{R}^n}|\langle\nabla f(x), u\rangle| dx$ is independent of $u$. Let $f$ be the characteristic function,  $\mathcal{X}_B$,
of the unit ball $B$. By the approximation argument in Zhang \cite{zhang}, we obtain
\begin{eqnarray*}\label{}
\int_{\mathbb{R}^n}|\langle\nabla \mathcal{X}_B(x), u\rangle| dx=\frac{2\omega_{n-1}}{n\omega_n}.
\end{eqnarray*}
Hence, the equality holds in (\ref{Petty projection inequality of log-concave functions 2}) when $f$ is the characteristic function of  a ball.

By  the definition of $\Pi^{\circ}f$,  integral of  polar coordinates with respect to the sphere Lebesgue measure,
the affine Sobolev inequality (which obtained by Zhang \cite[Theorem 1.1]{zhang}), we have
\begin{eqnarray}\label{Petty projection inequality of log-concave functions 2}
\int_{\mathbb{R}^n}\Pi^{\circ}f(x)dx&=&\int_{\mathbb{R}^n}\text{exp}\left[-\frac{1}{2}\int_{\mathbb{R}^n}|\langle\nabla \varphi, y\rangle| fdx\right]dy  \nonumber  \\
&=& \int_{0}^{\infty}r^{n-1} \int_{S^{n-1}}\text{exp}\left[-\frac{r}{2}\int_{\mathbb{R}^n}|\langle\nabla f(x), u\rangle| dx\right]dudr  \nonumber\\
&=& 2^{n} \int_{0}^{\infty}r^{n-1} e^{-r}dr \int_{S^{n-1}}\left[\int_{\mathbb{R}^n}|\langle\nabla f(x), u\rangle| dx\right]^{-n}du \nonumber\\
&=& 2^{n} \Gamma(n) \int_{S^{n-1}}\left[\int_{\mathbb{R}^n}|\langle\nabla f(x), u\rangle| dx\right]^{-n}du \nonumber\\
& \leq&   \Gamma(n+1) \left(\frac{\omega_n}{\omega_{n-1}}\right)^{n}\left( \int_{\mathbb{R}^n} |
f(x) |^{\frac n{n-1}} dx\right)^{1-n}.
\end{eqnarray}
Equality condition follows from the affine Sobolev inequality.
\end{proof}

As  mentioned in \cite[Corollary 3.1]{SMV} that the first inequality in Theorem \ref{Petty projection inequality of log-concave functions}
includes a geometric inequality,  which  involves the surface area  and the volume of the polar body of Petty  projection body.
The second  inequality in Theorem \ref{Petty projection inequality of log-concave functions} implies the Petty projection  inequality.

\section{The functional LYZ ellipsoid of log-concave functions}

The Minkowski functional of $K$  is defined as
\begin{eqnarray*}\label{}
\|x\|_K=\max\{t\geq 0:  x\in tK\}, \quad x\in \mathbb{R}^n.
\end{eqnarray*}
For $K\in \mathcal{K}_o^n$,
\begin{eqnarray*}\label{}
\|x\|_K=h_{K^{\circ}}(x),
\end{eqnarray*}
where %$x\in\mathbb{R}^n$ and
$K^{\circ}=\left\{x\in\mathbb{R}^n: \langle x, y \rangle\leq 1 \  \text{for all} \ y\in K  \right\}$ %denotes
is the polar body of $K$.

\begin{lemma}\label{include geometric case}
Let $K\in \mathcal{K}_o^n$. If $f(x)=e^{-\frac{\|x\|_K^2}{2}}$, then
\begin{eqnarray*}\label{}
\Gamma_{-2}f(x)=e^{-\frac{\|x\|_{\Gamma_{-2}K}^2}{2}}.
\end{eqnarray*}
\end{lemma}
\begin{proof}
Let $\bar{V}_K$ denote the normalized cone measure of $K$, which is given by
$$
d\bar{V}_{K}(z)=\frac{\langle z,\nu_{K}(z)\rangle}{nV(K)}d\mathcal{H}^{n-1}(z)\quad \text{for}\quad z\in \partial K.
$$
For $y\in \mathbb{R}^n$,  we write $y=rz$, with $z\in \partial K$, $dy=nV(K)r^{n-1}drd\bar{V}_{K}(z)$.
Since the map $y\mapsto \nabla \|y\|_{K}$ is 0-homogeneous and the fact that $\|z\|_K=1$ when $z\in\partial K$, therefore
\begin{eqnarray*}\label{}
&&\int_{\mathbb{R}^n}|\langle x, y\rangle|^2h_f(y)^{-1}d\mu_f(y)\\
&&\quad \quad =\int_{\mathbb{R}^n}|\langle x, \nabla(\|y\|_K^2\|/2)\rangle|^2\left(\frac{\|\nabla(\|y\|_K^2\|/2)\|_{K^{\circ}}^2}{2}\right)^{-1}
e^{-\frac{\|y\|_K^2}{2}}dy\\
&&\quad \quad =\int_{\mathbb{R}^n}|\langle x, \|y\|_K\nabla\|y\|_K\rangle|^2\left(\frac{\|\|y\|_K\nabla\|y\|_K\|_{K^{\circ}}^2}{2}\right)^{-1}
e^{-\frac{\|y\|_K^2}{2}}dy\\
&&\quad \quad =nV(K)\int_0^{\infty}\int_{\partial K}r^{n-1}|\langle x, \nabla\|z\|_K\rangle|^2\left(\frac{\|\nabla\|z\|_K\|_{K^{\circ}}^2}{2}\right)^{-1}
e^{-\frac{r^2}{2}}d\bar{V}_K(z)dr\\
&&\quad \quad =2^{\frac{n}{2}}\Gamma(\tfrac{n}{2})\int_{\partial K}|\langle x, \nu_K(z)\rangle|^2(\langle z, \nu_K(z)\rangle)^{-1}
d\mathcal{H}^{n-1}(z)\\
&&\quad \quad =2^{\frac{n}{2}}\Gamma(\tfrac{n}{2}) V(K)h_{\Gamma_{-2}^{\circ}K}^{2}(x).
\end{eqnarray*}
A direct calculation  shows

\begin{eqnarray}
J(e^{-\frac{\|x\|_K^2}{2}})=2^{\tfrac{n}{2}}\Gamma(\tfrac{n}{2}+1)V(K).
\end{eqnarray}
Hence
\begin{eqnarray*}
 -\log \Gamma_{-2}^{}f(x)=\frac{1}{2}\|x\|_{\Gamma_{-2}^{}K}^{2}.
\end{eqnarray*}
\end{proof}

Analogously,  $\Gamma_{-2}f$ is also an intertwining operator with the linear group ${\rm GL}(n)$.

\begin{lemma}
Let $f$ be an integrable log-concave function in $\mathbb{R}^n$. Then
\begin{eqnarray*}\label{}
\Gamma_{-2}(f\circ T)=(\Gamma_{-2}f)\circ T,
\end{eqnarray*}
for $T\in {\rm GL}(n)$.
\end{lemma}
\begin{proof}
Let $f=e^{-\varphi}$ and $T\in {\rm GL}(n)$. Since $\nabla_{y}(\varphi\circ T)=T^t\nabla_{Ty}\varphi$, so that
\begin{eqnarray*}\label{}
(\varphi\circ T)^{*}(y)&=&\sup_{x\in\mathbb{R}^n}\left\{\langle x,y\rangle-\varphi(Tx)\right\}\\
&=&\sup_{x\in\mathbb{R}^n}\left\{\langle x,T^{-t}y\rangle-\varphi(x)\right\}\\
&=&\varphi^{*}(T^{-t}y).
\end{eqnarray*}
Then,  by a simple calculation we have
\begin{eqnarray*}\label{}
-\log \Gamma_{-2}^{}f(x)&=&\frac{n}{4|\det T|J(f)}\int_{\mathbb{R}^n}|\langle x, \nabla(\varphi\circ T)(y)\rangle|^2\varphi^{*}(T^{-t}\nabla(\varphi\circ T)(y))^{-1}f(Ty)dy\\
&=&\frac{n}{4|\det T|J(f)}\int_{\mathbb{R}^n}|\langle x, T^t\nabla_{Ty}\varphi(Ty)\rangle|^2\varphi^{*}(\nabla_{Ty}\varphi(Ty))^{-1}f(Ty)dy\\
&=&\frac{n}{4J(f)}\int_{\mathbb{R}^n}|\langle Tx, \nabla\varphi(y)\rangle|^2\varphi^{*}(\nabla\varphi(y))^{-1}f(y)dy\\
&=&-\log \Gamma_{-2}^{}f(Tx).
\end{eqnarray*}
\end{proof}

%Combing the above two results, we conclude that:

The following result is the direct consequence.

\begin{cor}
For  $T\in {\rm GL}(n)$, the LYZ ellipsoid of Gauss function is
%itself, i.e.,
\begin{eqnarray}\label{}
\Gamma_{-2}(\gamma\circ T)(x)=e^{-\frac{\|Tx\|^2}{2}}.
\end{eqnarray}
%for $T\in {\rm GL}(n)$.
\end{cor}

The function

\begin{eqnarray}\label{new definition2}
x\mapsto \frac{n}{4 J(f)}\int_{\mathbb{R}^n}|\langle x,
y\rangle|^2h_f(y)^{-1}d\mu_f(y)
\end{eqnarray}
is a continuous convex function. The LYZ ellipsoid of log-concave
functions (\ref{definition}) can be rewritten as
\begin{eqnarray}\label{}
h_{\Gamma_{-2}^{\circ}f}(x)= \frac{n}{4
J(f)}\int_{\mathbb{R}^n}|\langle x, y\rangle|^2h_f(y)^{-1}d\mu_f(y).
\end{eqnarray}

%%%%%%%%%%%%%%%%%%%%%%%%%%%%%%%%%%%%%%%%%%%%%%%%%%%%%%%%%%%%%%%%%%%%%%%%%%%%%%%%%%%%%%%%%%%%%%%%%%

\end{document}